\documentclass[12pt,letterpaper,reqno]{amsart}
\usepackage[letterpaper,margin=1.2in,headheight=15pt]{geometry} 
\usepackage{graphicx,bbm}
\usepackage{amsmath, amssymb, amsfonts}
\usepackage[utf8]{inputenc}
\usepackage{epstopdf}
\usepackage{tikz-cd} 
\usepackage{setspace}
\usepackage{xcolor}
\usepackage{tcolorbox}
\usepackage[titletoc,title]{appendix}
\usepackage{enumitem}

\usepackage{hyperref}
\definecolor{darkred}{rgb}{0.5,0.15,0.15}
\hypersetup{colorlinks=true,urlcolor=darkred,linkcolor=darkred,citecolor=darkred}

\renewcommand{\ell}{X} 

\newcommand{\be}{\begin{eqnarray}}
\newcommand{\ee}{\end{eqnarray}}
\newcommand{\bea}{\begin{eqnarray}}
\newcommand{\eea}{\end{eqnarray}}

\newcommand{\ben}{\begin{eqnarray}}
\newcommand{\een}{\end{eqnarray}}

\theoremstyle{plain}
\newtheorem{thm}{Theorem}

\newtheorem{cor}[thm]{Corollary}

\theoremstyle{definition}

\newtheorem*{question*}{Question}

\theoremstyle{remark}

\numberwithin{equation}{section}

\title{Difference equation for the Gromov-Witten potential of the resolved conifold}

\author{Murad Alim}
\address{Fachbereich Mathematik, Universit\"at Hamburg, Bundesstr. 55, 20146, Hamburg}

\begin{document}


\begin{abstract}
A difference equation is proved for the Gromov-Witten potential of the resolved conifold. Using the Gopakumar-Vafa resummation of the Gromov-Witten invariants of any Calabi-Yau threefold, it is further shown that similar difference equations are satisfied by the part of the resummed potential containing the contribution of the genus zero GV invariants. 
\end{abstract}
\maketitle
\section{Introduction}
Gromov-Witten (GW) theory of a non-singular algebraic variety $X$ is concerned with the study of integrals over the moduli spaces of maps from Riemann surfaces into $X$. When $X$ is a point, GW theory has been related to $2d$ topological gravity and the KdV integrable hierarchy in the works of Witten \cite{Witten} and Kontsevich \cite{Kontsevich}. For $X=\mathbb{P}^1$ the relation of GW theory to the Toda integrable hierarchy was studied in \cite{Pandharipande} using a conjectured difference equation for the GW potential which implied a corresponding difference equation for Hurwitz numbers. The latter was proved in \cite{Okounkov}. GW theory for Calabi-Yau (CY) threefolds has greatly benefited from mirror symmetry, which was discovered in String Theory and which puts forward connections between symplectic and algebraic geometry. Within String Theory, GW theory is studied using topological strings, which offer many connections to Chern-Simons theory as well as matrix models. A general integrable hierarchy structure of topological string theory and in turn GW theory is expected \cite{ADKMV} and is subject of continued efforts. Recently the subject has been further stimulated from the combination of ideas from wall-crossing in Donaldson Thomas (DT) theory and the exact WKB analysis, most notably in works of Bridgeland \cite{BridgelandDT,BridgelandCon}. In particular, in \cite{BridgelandCon}, a tau function for the resolved conifold is put forward which solves a Riemann-Hilbert problem associated to DT wall-crossing while also containing the information of all genus GW invariants in an asymptotic expansion.

This note is concerned with the GW potential of the resolved conifold. In particular, a difference equation similar to the one conjecturally satisfied by the GW theory of $\mathbb{P}^1$ \cite{Pandharipande} is proved. The ideas needed for the proof are adapted from the work of Iwaki, Koike and Takei \cite{Iwaki} where a difference equation for the free energy of the Weber curve is proved. The latter is a generating function of the Euler characteristics of the moduli spaces of curves studied in \cite{HZ}. Similar difference equations appear in many interesting problems of mathematical physics such as for example in the study of matrix models and non-perturbative physics, see e.g. \cite[Eq.~4.134 and Sec.~5.1]{Marinolecture} and in the study of partition functions of supersymmetric $\mathcal{N}=2$ theories where they characterize the perturbative piece of the partition function, see e.g. \cite[Appendix A]{NO} as well as in the form of q-deformed Painlev\'e equations satisfied by suitably constructed topological string partition functions, see \cite{BGT} and references therein. 

\subsection*{Acknowledgements}
The author would like to thank Florian Beck, Arpan Saha and J\"org Teschner for comments on the draft as well as the LMU in Munich for hospitality. This work is supported through the DFG Emmy Noether grant AL 1407/2-1.

\section{Gromov-Witten potentials}
Let $X$ be a Calabi-Yau threefold. The GW potential of $X$ is the following formal power series:
\begin{equation}
F(\lambda,t) = \sum_{g\ge 0}  \lambda^{2g-2} F^g(t)= \sum_{g\ge 0}  \lambda^{2g-2} \sum_{\beta\in H_2(X,\mathbb{Z})}  N^g_{\beta} \,q^{\beta}\, ,
\end{equation}
where $q^{\beta} := \exp (2\pi i t^{\beta})$ is a formal variable living in a suitable completion of the effective cone in the group ring of $H_2(X,\mathbb{Z})$, $\lambda$ is a formal parameter corresponding to the topological string coupling and $N^g_{\beta}$ are the GW invariants. 

The GW potential can be written as:
\begin{equation}
F=F_{\beta=0} + \tilde{F}\,,
\end{equation}
where $F_{\beta=0}$ denotes the contribution from constant maps and $ \tilde{F}$ the contribution from non-constant maps. The constant map contribution at genus 0 and 1 are $t$ dependent and the higher genus constant map contributions take the universal form \cite{Faber}:
\begin{equation}
F_{\beta=0}^g = \frac{\chi(X)(-1)^{g-1}\, B_{2g}\, B_{2g-2}}{4g (2g-2)\, (2g-2)!}\,, \quad g\ge2\,,
\end{equation}
where $\chi(X)$ is the Euler character of $X$.
This note is concerned with the GW potential of the CY threefold given by the total space of the rank two bundle over the projective line:
\begin{equation}
\mathcal{O}(-1) \oplus \mathcal{O}(-1) \rightarrow \mathbb{P}^1\,,
\end{equation}
which corresponds to the resolution of the conifold singularity in $\mathbb{C}^4$ and is known as the resolved conifold. The GW potential for this geometry was determined in physics \cite{Gopakumar:1998ii,GV}, and in mathematics \cite{Faber} with the following outcome for the non-constant maps:\footnote{See also \cite{MM} for the determination of $F^g$ from a string theory duality and the explicit appearance of the polylogarithm expressions.}
\begin{equation}
\tilde{F}^0= \textrm{Li}_{3}(q)\,, \quad \tilde{F}^g=\frac{(-1)^{g-1}B_{2g}}{2g (2g-2)!} \textrm{Li}_{3-2g} (q)\,, \quad g\ge1\,,
\end{equation}
where $q:=\exp(2\pi i \,t)$ and the polylogarithm ist defined by:
\begin{equation}
\textrm{Li}_s(z) = \sum_{n=0}^{\infty} \frac{z^n}{n^s}\, ,\quad s\in \mathbb{C}\,.
\end{equation}


\section{A difference equation for the Gromov-Witten potential}
\begin{thm} \label{diffeq} The contribution of the non-constant maps $\tilde{F}(\lambda,t)$ to the GW potential of the resolved conifold satisfies the following difference equation:
\begin{equation}
\tilde{F}(\lambda,t+\check{\lambda}) - 2 \tilde{F}(\lambda,t) + \tilde{F}(\lambda,t-\check{\lambda})= \left(\frac{1}{2\pi }\frac{\partial}{\partial t}\right)^2\, \tilde{F}^0(t) \,,\quad \check{\lambda}=\frac{\lambda}{2\pi}\,.
\end{equation}
\end{thm}

\begin{proof}
Starting from $\textrm{Li}_1(q)=-\log(1-q)$ and using the property:
\begin{equation} \label{polylogder}
\theta_q \textrm{Li}_s(q) =\textrm{Li}_{s-1} (q)\,,  \quad \theta_q:= q \,\frac{d}{dq}\,,
\end{equation}
we write 
\begin{equation}
\tilde{F}^g=\frac{(-1)^{g-1}B_{2g}}{2g (2g-2)!} \,\theta_q^{2g-2} \textrm{Li}_1(q)\,, \quad g\ge1\,.
\end{equation}
In the following, the proofs of \cite[Thm.~4.7 and 4.9]{Iwaki} are adapted to the current setting. Consider the generating function of Bernoulli numbers:
\begin{equation}
\frac{w}{e^w-1} = \sum_{n=0}^{\infty} B_n \frac{w^n}{n!}\,.
\end{equation}
Applying $w \frac{d}{dw}$ to both sides and rearranging gives:
\begin{equation}
\frac{w^2 e^w}{(e^w-1)^2} = B_0 - \sum_{n=2}^{\infty} \frac{B_n}{n (n-2)!} w^n = 1 -\sum_{g=1}^{\infty} \frac{B_{2g}}{2g (2g-2)!} w^{2g}\,,
\end{equation}
where the last equality is obtained by noting that all $B_{2n+1}, n\in \mathbb{N}\setminus\{0\}$ vanish. This yields the following:
\begin{equation}
(e^w-2+e^{-w})  \left(\frac{1}{w^2} -\sum_{g=1}^{\infty} \frac{B_{2g}}{2g (2g-2)!} w^{2g-2} \right) = 1\,.
\end{equation}
In the next step, we replace on both sides of this equation the variable $w$ by an operator acting on functions of $t$ namely:
$$ w \rightarrow \check{\lambda} \frac{\partial}{\partial t}= i \lambda \,\theta_q\,. $$
Acting with both sides on $\textrm{Li}_1(q)$ we obtain:
\begin{equation}
(e^{\check{\lambda} \partial_t}-2\cdot \textrm{id}+e^{-\check{\lambda}\partial_t})  \left(- \lambda^{-2} \theta_q^{-2} -\sum_{g=1}^{\infty} \frac{(-1)^{g-1}B_{2g}}{2g (2g-2)!}  \lambda^{2g-2}\,\theta_q^{2g-2} \right) \textrm{Li}_1(q) = \textrm{id}\cdot \textrm{Li}_1(q) \,,
\end{equation}
by using $\theta_q^{2} \tilde{F}^0=  \theta_q^{2} \textrm{Li}_3(q)= \textrm{Li}_1(q) $ and interpreting $\theta_q^{-1}$ as an anti-derivative, we obtain:
\begin{equation}
(e^{\check{\lambda} \partial_t}-2\cdot \textrm{id}+e^{-\check{\lambda}\partial_t})  \tilde{F}(\lambda,t) = -\theta_q^2 \tilde{F}^0(q) \,,
\end{equation}
which proves the theorem.
\end{proof}

\begin{cor}
For every $g\ge 1$ the difference equation gives a recursive differential equation which determines $\frac{\partial^2}{\partial t^2}\tilde{F}^g(t)\, \, g \ge 1$ by:
\begin{equation}
\sum_{k=0}^{g} \frac{1}{(2g-2k+2)!} \left(\frac{1}{2\pi} \frac{\partial}{\partial t}\right)^{2g-2k+2} \tilde{F}^k(t) =0 \,, \quad g\ge1\,.
\end{equation}
\end{cor}
\begin{proof}
This follows from expanding the L.H.S. of the theorem in $\lambda$ and then matching the coefficients of $\lambda^{2g}$ on both sides.
\end{proof}

\section{Difference equation for Gopakumar-Vafa resummation}
The statement of the theorem is in the same spirit as the conjectured difference equation \cite{Pandharipande} for GW theory of $\mathbb{P}^1$ except that here only the genus zero part of the GW potential appears on the R.H.S. We expect this result to shed further light on the relation of GW theory of the resolved conifold and the tau function studied in \cite{BridgelandCon}. Furthermore, it may be useful for proving a conjecture of Brini \cite{Brini} relating the GW theory of the resolved conifold to the Ablowitz-Ladik integrable hierarchy.\footnote{Note added in the revision: The proof of Brini's conjecture as well as the relation to the special functions appearing in Bridgeland's tau function has now appeared in Ref.~\cite{AS}.} Moreover, the result may offer new insights into the universal structure of GW theory on CY threefolds. An indication of this is given in the following corollary, which requires the Gopakumar-Vafa resummation of the GW potential \cite{Gopakumar:1998ii}. The latter conjectures that there exist $n^g_{\beta} \in \mathbb{Z}$ (called GV invariants), such that for any CY threefold $X$, the GW potential can be written as:
\begin{equation}
\tilde{F}(\lambda,t)= \sum_{\beta>0}\sum_{g\ge 0} n^g_{\beta}\, \sum_{k\ge 1} \frac{1}{k} \left( 2 \sin \left( \frac{k\lambda}{2}\right)\right)^{2g-2} q^{k\beta}\,.
\end{equation}
We focus on the contribution of the genus zero GV invariants to the resummed GW potential, which we denote by:
$$ \tilde{F}_{GV_0}(\lambda,t)=\sum_{\beta>0} n^0_{\beta}\, \sum_{k\ge 1} \frac{1}{k} \left( 2 \sin \left( \frac{k\lambda}{2}\right)\right)^{-2} q^{k\beta} \,.$$

\begin{cor}
Fix $\alpha \in H_{2}(X,\mathbb{Z}), \alpha>0$, then $\tilde{F}_{GV_0}$ satisfies the following difference equation in the corresponding formal variable $t^{\alpha}$:
\begin{equation}
\tilde{F}_{GV_0}(\lambda,t^{\circ},t^{\alpha}+\check{\lambda}) - 2 \tilde{F}_{GV_0}(\lambda,t) + \tilde{F}_{GV_0}(\lambda,t^{\circ},t^{\alpha}-\check{\lambda}) =   \left(\frac{1}{2\pi }\frac{\partial}{\partial t^{\alpha}}\right)^2\, \tilde{F}^0(t) \,,\quad \check{\lambda}=\frac{\lambda}{2\pi}\,,
\end{equation}
where $t$ refers to the set of all formal variables $t^{\beta}$, $t^{\circ}=t\setminus \{ t^{\alpha}\}$ and
$$ \tilde{F}^0(t)=\sum_{\beta>0} n_{\beta}^0 \textrm{Li}_3(q^{\beta})\,, \quad q^{\beta}= \exp(2\pi i t^{\beta})\,.$$ 
\end{cor}

\begin{proof}
We use the Laurent expansion:
\begin{equation}
 \left( 2 \sin \left( \frac{s}{2}\right)\right)^{-2}= \frac{1}{(s)^2} - \frac{1}{12} + \sum_{g\ge 2} \frac{(-1)^{g-1}B_{2g}}{2g (2g-2)!} s^{2g-2}\,,
 \end{equation}
to write:
\begin{eqnarray}
 \tilde{F}_{GV_0}(\lambda,t)&=&\sum_{\beta>0, \beta \ne \alpha} n_{\beta}^0  \left( \lambda^{-2}\,\textrm{Li}_3(q^{\beta}) + \sum_{g=1}^{\infty} \lambda^{2g-2}   \frac{(-1)^{g-1}B_{2g}}{2g (2g-2)!} \textrm{Li}_{3-2g} (q^{\beta}) \right)\, \nonumber \\
 &+& n_{\alpha}^0  \left( \lambda^{-2}\,\textrm{Li}_3(q^{\alpha}) + \sum_{g=1}^{\infty} \lambda^{2g-2}   \frac{(-1)^{g-1}B_{2g}}{2g (2g-2)!} \textrm{Li}_{3-2g} (q^{\alpha}) \right)\,.
 \end{eqnarray}
The proof of Thm.\ref{diffeq} can be repeated for the second part on the R.H.S. The statement of the corollary follows by noting that the operator giving the difference equation in $t^{\alpha}$ gives zero when acting on the first, $t^{\alpha}$ independent, part.
 \end{proof}



\newcommand{\etalchar}[1]{$^{#1}$}

\end{document}